\documentclass[11pt,a4paper]{article}

\usepackage{graphicx,latexsym,euscript,makeidx,color,bm}
\usepackage{amsmath,amsfonts,amssymb,amsthm,thmtools,mathtools,mathrsfs,enumerate}
\usepackage[colorlinks,linkcolor=blue,citecolor=red]{hyperref}

\usepackage{geometry}
\allowdisplaybreaks[4]
   \def\cA{{\cal A}}  
     
     \def\tX{\ti X}
     \def\tZ{\ti Z}
\def\dbE{\mathbb{E}}     \def\tJ{\ti J}
\def\dbF{\mathbb{F}} \def\sF{\mathscr{F}}    
\def\dbG{\mathbb{G}} \def\sG{\mathscr{G}}  
\def\dbH{\mathbb{H}}   
   
  \def\hJ{\hat{J}}

 \def\sN{\mathscr{N}}  
   
\def\dbP{\mathbb{P}}   
   
\def\dbR{\mathbb{R}}   
\def\dbS{\mathbb{S}}   
   
   \def\cU{{\cal U}}

  \def\hX{\hat{X}} 
 \def\sY{\mathscr{Y}}  
 \def\sZ{\mathscr{Z}} \def\hZ{\hat{Z}} 

   \def\lt{\left}          \def\hb{\hbox}
\def\ms{\medskip}     \def\rt{\right}         
\def\h{\hat}          \def\lan{\langle}       \def\as{\text{a.s.}}
\def\q{\quad}         \def\ran{\rangle}       \def\tr{\hb{tr$\,$}}
\def\qq{\qquad}             
\def\no{\noindent}          
\def\hp{\hphantom}         
\def\nn{\nonumber}         
\def\rf{\eqref}       \def\Blan{\Big\lan\!\!} 
\def\cd{\cdot}        \def\Bran{\!\!\Big\ran} 
\def\deq{\triangleq}  \def\({\Big(}           
\def\ti{\tilde}       \def\){\Big)}           
   \def\[{\Big[}           \def\ts{\textstyle}
  \def\]{\Big]}           
\def\nid{\,|\,}       
\def\bid{\,\big|\,}   
\def\Bid{\,\Big|\,}

       \def\l{\lambda}    \def\D{\varDelta}
        \def\t{\tau}       
\def\d{\delta}       \def\th{\theta}    
      \def\L{\varLambda}
\def\f{\varphi}             \def\Om{\varOmega}
       \def\si{\sigma}    \def\Si{\varSigma}
\def\i{\infty}       \def\z{\zeta}      \def\Th{\Theta}

\newtheoremstyle{thry}
{}      
{}      
{\sl}   
{}      
{\bf}   
{.}     
{.5em}  
{}      
\theoremstyle{thry}

\newtheorem{theorem}{Theorem}[section]
\newtheorem{proposition}[theorem]{Proposition}
\newtheorem{corollary}[theorem]{Corollary}
\newtheorem{lemma}[theorem]{Lemma}

\theoremstyle{definition}

\newtheorem{notation}[theorem]{Notation}

\theoremstyle{remark}

\def\punct{}
\newtheoremstyle{dotless}{}{}{\rm}{}{\bf}{\punct}{.5em}{}
\theoremstyle{dotless}
\newtheorem{taggedthmx}{}
\newenvironment{taggedthm}[1]
 {\renewcommand\thetaggedthmx{#1}\taggedthmx}
 {\endtaggedthmx}
\newtheorem{taggedassumptionx}{}
\newenvironment{taggedassumption}[1]
 {\renewcommand\thetaggedassumptionx{#1}\taggedassumptionx}
 {\endtaggedassumptionx}

\makeatletter
   
   \@addtoreset{equation}{section}
   \newcommand{\setword}[2]{%
   \phantomsection
   #1\def\@currentlabel{\unexpanded{#1}}\label{#2}%
   }
\makeatother

\begin{document} 

\title{\bf Stochastic Linear-Quadratic Optimal Control with Partial Observation}
\author{Jingrui Sun\thanks{Department of Mathematics, Southern University of Science and Technology, Shenzhen,
                           518055, China (Email: {\tt sunjr@sustech.edu.cn}).
                           This author is supported by NSFC grant 11901280 and Guangdong Basic and Applied Basic Research
                           Foundation 2021A1515010031.}
~~~and~~
Jie Xiong\thanks{Department of Mathematics and SUSTech International center for Mathematics, Southern University
                 of Science and Technology, Shenzhen, 518055, China (Email: {\tt xiongj@sustech.edu.cn}).
                 This author is supported by NSFC Grants 61873325 and 11831010.}
}

\maketitle

\no{\bf Abstract.}
The paper studies a class of quadratic optimal control problems for partially observable linear dynamical systems.
In contrast to the full information case, the control is required to be adapted to the filtration
generated by the observation system, which in turn is influenced by the control. 
The variation method fails in this case due to the fact that the filtration is not fixed.
To overcome the difficulty, we use the orthogonal decomposition of the state process to write the cost functional
as the sum of two parts: one is a functional of the control and the filtering process and the other part is
independent of the choice of the control.
The first part possesses a mathematical structure similar to the full information problem.
By completing the square, it is shown that the optimal control is given by a feedback representation via the filtering process.
The optimal value is also obtained explicitly.

\ms
\no{\bf Key words.}
optimal control, linear-quadratic, observation process, filtering, Riccati equation.

\ms
\no{\bf AMS 2020 Mathematics Subject Classification.}
49N10, 49N30, 93E11, 93E20.

\section{Introduction}\label{Sec:Intro}

Let $(\Om,\sF,\dbP)$ be a complete probability space on which two standard independent Brownian motions
$W=\{W(t)=(W_1(t),\ldots,W_d(t))^\top;\,t\ge0\}$ and $W'=\{W'(t)=(W'_1(t),\ldots,W'_k(t))^\top;\,t\ge0\}$,
with values in $\dbR^d$ and $\dbR^k$, respectively, are defined.
The superscript $\top$ denotes the transpose of a vector or a matrix, so both $W(t)$ and $W'(t)$ are
column vectors.
Let $\dbF=\{\sF_t\}$ be the usual augmentation of the natural filtration generated by $(W,W')$,
and let $\dbG=\{\sG_t\}$ be a sub-filtration of $\dbF$.
For a random variable $\xi$, we write $\xi\in\sG_t$ if $\xi$ is $\sG_t$-measurable;
and for a stochastic process $\f$, we write $\f\in\dbG$ if it is progressively measurable with respect
to the filtration $\dbG$.


\ms

Consider the following linear stochastic differential equation (SDE) over a finite horizon $[0,T]$:
\begin{equation}\label{state}\left\{\begin{aligned}
dX(t) &= [A(t)X(t) + B(t)u(t) +a(t)]dt + C(t)dW(t) + D(t)dW'(t), \\
 X(0) &= x,
\end{aligned}\right.\end{equation}
where the {\it initial state} $x$ is a constant vector in $\dbR^n$, and the coefficients $A$, $B$, $a$, $C$, and $D$
are deterministic, bounded functions on $[0,T]$, with values in $\dbR^{n\times n}$, $\dbR^{n\times m}$, $\dbR^n$, $\dbR^{n\times d}$,
and $\dbR^{n\times k}$, respectively. For a {\it control} $u$ that belongs to the space
\begin{align}\label{def-cU}
\ts \cU=\Big\{ u:[0,T]\times\Om\to\dbR^m \bigm| u\in\dbF~\text{and}~\dbE\int_0^T|u(t)|^2dt<\i \Big\},
\end{align}
the {\it state equation} \rf{state} admits a unique strong solution $X=\{X(t);\,0\le t\le T\}$,
which is a square-integrable $\{\sF_t\}$-semimartingale satisfying
\begin{equation}\label{Bound:XY}
\dbE\lt[\sup_{0\le t\le T}|X(t)|^2\rt]<\i.
\end{equation}
The classical {\it linear-quadratic (LQ) optimal control problem} is to find a control $u^*\in\cU$ that minimizes the
quadratic {\it cost functional}
\begin{align}\label{cost}
J(x;u) &= \dbE\Big\{\lan GX(T),X(T)\ran + 2\lan g,X(T)\ran  \nn\\
&\hp{=\ } +\int_0^T\bigg[\Blan\begin{pmatrix*}[l]Q(t) & \!S(t)^\top \\ S(t) & \!R(t)\end{pmatrix*}\!
                              \begin{pmatrix}X(t) \\ u(t)\end{pmatrix}\!,
                              \begin{pmatrix}X(t) \\ u(t)\end{pmatrix}\Bran  \nn\\
&\hp{=\ } +2\Blan\begin{pmatrix}q(t) \\ r(t)\end{pmatrix}\!,
                 \begin{pmatrix}X(t) \\ u(t)\end{pmatrix}\Bran\bigg] dt\bigg\}
\end{align}
over $\cU$, subject to the state equation \rf{state}.
In \rf{cost}, $\lan\cd\,,\cd\ran$ is the Frobenius inner product;
$G$ is a symmetric $n\times n$ constant matrix; $g\in\dbR^n$ is a constant vector;
$Q,S,R$ are bounded deterministic matrix-valued functions of proper dimensions over $[0,T]$
such that the blocked matrix in the Lebesgue integral is symmetric;
$q,r$ are bounded deterministic functions over $[0,T]$, with values in $\dbR^n$ and $\dbR^m$, respectively.

\ms

The above LQ problem can be elegantly solved by a Riccati equation approach, whose optimal control turns out
to be a linear feedback of the current state (see, for example, \cite{Yong-Zhou1999,Sun-Yong2020}).
Thus, the optimal control could be constructed if the controller has access to the exact value of the state.
However, in many practical situations it often happens that some components of the state may be inaccessible
for observation, and there could be noise existing in the observation systems.
In this case only partial information is available to the controller, and the control has to be selected according
to the information provided by the observation systems.

\ms

In this paper we consider the case that the observation process evolves according to the following SDE:
\begin{equation}\label{SDE:observer}\left\{\begin{aligned}
dY(t) &= [H(t)X(t)+h(t)]dt + K(t)dW(t), \q 0\le t\le T,\\
 Y(0) &= 0,
\end{aligned}\right.\end{equation}
where $H$, $h$, and $K$ are deterministic, bounded functions on $[0,T]$, with values in $\dbR^{d\times n}$, $\dbR^d$,
and $\dbR^{d\times d}$, respectively.
For each control $u\in\cU$, the observation process $Y=\{Y(t);\,0\le t\le T\}$ is a square-integrable
$\{\sF_t\}$-semimartingale.
Let $\sY^u=\{\sY^u_t\}_{0\le t\le T}$ be the usual augmentation of the filtration generated by $Y$.
Clearly, the filtration $\sY^u$ depends on the choice of $u$.
We say that a control $u\in\cU$ is {\it admissible} if it is progressively measurable with respect
to $\sY^u$; that is, the set of admissible controls is
\begin{align}\label{def-Uad}
\ts \cU_{ad}=\Big\{ u:[0,T]\times\Om\to\dbR^m \bigm| u\in\sY^u~\hb{and}~\dbE\int_0^T|u(t)|^2dt<\i \Big\}.
\end{align}
At each time $t$, the controller first observes $Y(t)$, then applies $u(t)$ to the state system \rf{state} immediately.
The objective of our control problem is to choose an admissible control such that the cost functional \rf{cost} is minimized.

\begin{taggedthm}{Problem (O).}
For a given initial state $x\in\dbR^n$, find a control $v\in\sY^v$ such that
\begin{align*}
J(x;v)=\inf_{u\in\sY^u}J(x;u) \equiv V(x).
\end{align*}
\end{taggedthm}

The process $v$, if exists, is called an {\it optimal control} for the initial state $x$,
and the function $V$ is called the {\it value function} of Problem (O).

\ms

In contrast to the completely observable case, the essential difficulty is that the filtration $\sY^u$ is not fixed
(depending on the control $u$) and the linear structure of the admissible control set is thereby corrupted.
For this reason it is hard to derive the optimality system by the variational method.
In fact, for a given control $u\in\cU$, it is not even easy to decide whether it is admissible or not.
To overcome this difficulty, one way is to apply the separation principle to decouple the problems of
optimal control and state estimation.
This method is based on an additional requirement on the admissible control, i.e., $u\in\cU$ is said to be admissible
if, in addition to $u\in\sY^u$, $u$ is adapted to the smaller filtration $\sY^0$; see, for example,
Wonham \cite{Wonham1968} and Bensoussan \cite{Bensoussan2018}.
Along this line, many research on LQ optimal control of partially observable systems were carried out in recent years,
among which we would like to mention the works \cite{Huang-Wang-Xiong2009,Huang-Wang-Zhang2020,Wang-Wang-Yan2021}
on backward stochastic control systems under partial information,
the works \cite{Shi-Wu2010,Shi-Zhu2013,Wang-Wu-Xiong2015} on optimal control problems of forward-backward stochastic differential
equations (FBSDEs) with partial information,
the works \cite{Wang-Yu2012,Wu-Zhuang2018} on differential games with partially observable systems.
Another way of analyzing the problem is to convert it into a completely observable stochastic optimal
control problem by Girsanov's transformation (see, for example, Wang, Wu, and Xiong \cite{Wang-Wu-Xiong2013},
Yong and Zhou \cite[Chapter 2]{Yong-Zhou1999}, or the references therein).
However, this approach turns the problem into an infinite-dimensional one that is also difficult to solve
and requires higher integrability of the state process which cannot be fulfilled in the LQ case in general.

\ms

In this paper we introduce a new method of solving Problem (O) {\it without imposing additional requirements on the admissible control}.
The idea is to first derive the filtering equation for a fixed admissible control and then use the orthogonal decomposition of
the state process to write the cost functional as the sum of two independent parts.
One part is a functional of the admissible control and the filtering process,
and the other part is a functional of the estimate error that is independent of the choice of the admissible control.
The first part possesses a mathematical structure similar to the full information problem, which we solve by completing the square.
The second part is simplified using integration by parts.
Finally, we show by verification that the optimal control is given by a feedback representation via the filtering process.
We also obtain the optimal value explicitly.

\ms

The rest of the paper is organized as follows.
In \autoref{Sec:Pre} we collect some notation and preliminary results that we need later.
Then we present the main results of the paper in \autoref{sec:main-results}.
The subsequent two sections are devoted to the proofs of the main results.
We derive the filtering equation for a fixed admissible control in \autoref{sec:filter}
and construct the optimal control in \autoref{sec:optimal-u}.
Finally, we give some concluding remarks in \autoref{sec:conclusion}.

\section{Preliminaries}\label{Sec:Pre}

We begin by introducing some notation.
All vectors in the paper are column vectors, and the transpose is denoted by the superscript $\top$.
The Euclidean space $\dbR^{n\times m}$ of $n\times m$ real matrices is equipped with the Frobenius inner product
$$ \lan M,N\ran=\tr(M^\top N), \q M,N\in\dbR^{n\times m}, $$
where $\tr(M^\top N)$ is the trace of $M^\top N$.
The norm induced by the Frobenius inner product is denoted by $|\cd|$.
The identity matrix of size $n$ is denoted by $I_n$.
For a subset $\dbH$ of $\dbR^{n\times m}$, we denote by $C([0,T];\dbH)$ the space of continuous functions
from $[0,T]$ into $\dbH$, and by $L^\i(0,T;\dbH)$ the space of Lebesgue measurable, essentially bounded functions
from $[0,T]$ into $\dbH$.
Let $\dbS^n$ be the subspace of $\dbR^{n\times n}$ consisting of symmetric matrices.
For $\dbS^n$-valued functions $M$ and $N$, we write $M\ge N$ (respectively, $M>N$) if $M-N$ is positive
semidefinite (respectively, positive definite) almost everywhere with respect to the Lebesgue measure.

\ms

Throughout this paper we impose the following assumptions.

\begin{taggedassumption}{(A1)}\label{A1}
The deterministic functions
\begin{align*}
& A:[0,T]\to\dbR^{n\times n}, \q B:[0,T]\to\dbR^{n\times m}, \q a:[0,T]\to\dbR^n,  \q C:[0,T]\to\dbR^{n\times d},   \\
& D:[0,T]\to\dbR^{n\times k}, \q H:[0,T]\to\dbR^{d\times n}, \q h:[0,T]\to\dbR^d,  \q K:[0,T]\to\dbR^{d\times d}
\end{align*}
are Lebesgue measurable and bounded on $[0,T]$.
\end{taggedassumption}
\begin{taggedassumption}{(A2)}\label{A2}
$K(t)$ is invertible for a.e. $t\in[0,T]$, and $K^{-1}$ is bounded on $[0,T]$.
\end{taggedassumption}
\begin{taggedassumption}{(A3)}\label{A3}
$G\in\dbS^n$, $g\in\dbR^n$, and the deterministic functions
$$ Q:[0,T]\to\dbS^n, ~ S:[0,T]\to\dbR^{m\times n}, ~ R:[0,T]\to\dbS^m, ~  q:[0,T]\to\dbR^n, ~ r:[0,T]\to\dbR^m $$
are Lebesgue measurable and bounded on $[0,T]$, satisfying
$$ G\ge 0, \q R(t)\ge\d I_m, \q Q(t)-S(t)^\top R(t)^{-1}S(t)\ge0, \q \as~t\in[0,T], $$
where $\d>0$ is some constant.
\end{taggedassumption}

The following result is standard in the LQ theory. For a quick proof, see Yong and Zhou \cite[pp. 297--298]{Yong-Zhou1999}.

\begin{lemma}\label{lmm:LQ-theory}
Let \ref{A1} and \ref{A3} hold. Then the Riccati equation
\begin{equation}\label{Ric:P}\left\{\begin{aligned}
& \dot P(t)+P(t)A(t)+A(t)^\top P(t) + Q(t) \\
&\hp{\dot P(t)} -[P(t)B(t)+S(t)^\top]R(t)^{-1}[B(t)^\top P(t)+S(t)]=0,  \\
& P(T)=G
\end{aligned}\right.\end{equation}
admits a unique solution $P\in C([0,T];\dbS^n)$, which is positive semidefinite everywhere on $[0,T]$.
\end{lemma}

For a process $Z=\{Z(t);\,0\le t\le T\}$, let
\begin{eqnarray*}
& \sZ_t^\circ \deq \si(Z(s);\, 0\le s\le t), \q 0\le t\le T, \\
& \sN^{_Z} \deq \{N\subseteq \Om;\, \exists\, G \in \sZ_T^\circ \text{ with } N\subseteq G \text{ and } \dbP(G)=0 \},
\end{eqnarray*}
and define
$$  \sZ_T\deq \si\big(\sZ_T^\circ,\sN^{_Z}\big),
\q \sZ_t \deq \bigcap_{s>t}\si\big(\sZ_s^\circ,\sN^{_Z}\big) = \si\big(\sZ_{t+}^\circ,\sN^{_Z}\big), \q 0\le t<T. $$
The filtered space $(\Om,\{\sZ_t\},\sZ_T,\dbP)$ satisfies the usual conditions.
We call $\{\sZ_t\}$ the {\it usual augmentation} of $\{\sZ_t^\circ\}$, the filtration generated by $Z$.

\ms

Let $\dbG=\{\sG_t\}_{t\ge0}$ be a sub-filtration of $\dbF=\{\sF_t\}$, which also satisfies the usual conditions.
Consider an integrable process $Z=\{Z(t);\,0\le t\le T\}$.
The conditional expectation
$$ \hZ(t) = \dbE[Z(t)\nid \sG_t] $$
defines $\hZ(t)$ up to a zero probability set but does not give us the paths of $\hZ$, which requires specifying
its values simultaneously at the uncountable set of times in $[0,T]$.
We need a good version for $\hZ$ that is at least progressively measurable with respect to $\{\sG_t\}$.
Such a version is referred to as the $\{\sG_t\}$-{\it optional projection} of $Z$, which exists under various conditions.
For our purpose, we shall focus on the case that $Z=\{Z(t);\,0\le t\le T\}$ is a continuous, square-integrable
$\{\sF_t\}$-semimartingale. We have the following result.

\begin{proposition}
Let $Z=\{Z(t);\,0\le t\le T\}$ be a continuous, square-integrable $\{\sF_t\}$-semimartingale.
Then there exists a $\{\sG_t\}$-adapted process $\hZ=\{\hZ(t);\,0\le t\le T\}$ whose sample paths are right-continuous
with finite left-hand limits (RCLL), such that
$$ \hZ(t) = \dbE[Z(t)\nid \sG_t], \q\as ~\forall\, 0\le t\le T. $$
\end{proposition}

\begin{proof}
Since $Z$ is a continuous, square-integrable $\{\sF_t\}$-semimartingale, it admits the decomposition
$$ Z(t) = \th_1(t) - \th_2(t) + \eta(t), $$
where $\th_i=\{\th_i(t);\,0\le t\le T\}$, $i=1,2$, are continuous, $\{\sF_t\}$-adapted increasing processes,
and $\eta=\{\eta(t);\,0\le t\le T\}$ is a continuous $\{\sF_t\}$-martingale. Let
$$ \h\th_i(t) = \dbE[\th_i(t)\nid \sG_t], \q \h\eta(t) = \dbE[\eta(t)\nid \sG_t]. $$
For $0\le s<t\le T$,
\begin{eqnarray*}
& \dbE[\h\th_i(t)\nid \sG_s] =  \dbE[\th_i(t)\nid \sG_s] \ge \dbE[\th_i(s)\nid \sG_s] = \h\th_i(s), \\
& \dbE[\h\eta(t) \nid \sG_s]  =  \dbE[\eta(t)\nid \sG_s] = \dbE\big\{\dbE[\eta(t)\nid\sF_s]\nid \sG_s\big\}
                              =  \dbE[\eta(s)\nid \sG_s] = \h\eta(s).
\end{eqnarray*}
Thus, both $\{\h\th_i(t);\,0\le t\le T\}$ and $\{\h\eta(t);\,0\le t\le T\}$ are $\{\sG_t\}$-submartingales.
Moreover, the functions
$$ t\mapsto \dbE[\h\th_i(t)]= \dbE[\th_i(t)], \q t\mapsto \dbE[\h\eta(t)]= \dbE[\eta(t)]$$
are continuous. Because the filtration $\{\sG_t\}$ satisfies the usual conditions,
the processes $\h\th_i$ ($i=1,2$) and $\h\eta$ have $\{\sG_t\}$-adapted, RCLL modifications
(see Karatzas and Shreve \cite{Karatzas-Shreve1991}, Theorem 1.3.13), and therefore so is $\dbE[Z(t)\nid\sG_t]$.
\end{proof}

\begin{notation}
Hereafter, we shall use the notation $\hZ=\{\hZ(t);\,0\le t\le T\}$ for the optional projection of
a process $Z=\{Z(t);\,0\le t\le T\}$, and use $\tZ=\{\tZ(t);\,0\le t\le T\}$ for the difference
$Z-\hZ$.
\end{notation}

When $Z$ is square-integrable, for each $t\in[0,T]$, $\dbE[Z(t)\nid\sG_t]$ is the orthogonal projection
of $Z(t)$ onto $L^2(\Om,\sG_t)$, the Hilbert space of $\sG_t$-measurable, square-integrable random variables.
In this case, $\tZ(t)=Z(t)-\hZ(t)$ is independent of $\sG_t$, and
$$ \dbE|Z(t)|^2 = \dbE|\hZ(t)|^2 + \dbE|\tZ(t)|^2. $$

\section{Main results}\label{sec:main-results}

In this section, we present the main results of the paper and briefly illustrate the idea.
The rigorous proofs are deferred to the subsequent sections.

\ms

As mentioned previously, the main difficulty of solving Problem (O) is that the filtration generated
by the observation process depends on the control, in which case the traditional variation method fails.
To deal with this difficulty, we shall first derive, for a fixed admissible control $u\in\cU_{ad}$,
the filtering equation for
$$ \hX(t)=\dbE[X(t)\nid\sY_t^u], $$
where $\{\sY_t^u\}$ is the usual augmentation of the filtration generated by the observation process $Y$ (depending on $u$).
The following ordinary differential equation (ODE) will be involved:
\begin{equation}\label{Si}\left\{\begin{aligned}
\dot\Si(t) &= [A(t)-C(t)K(t)^{-1}H(t)]\Si(t) + \Si(t)[A(t)-C(t)K(t)^{-1}H(t)]^\top  \\
&\hp{=\ } -\Si(t)H(t)^\top N(t)^{-1}H(t)\Si(t) + M(t), \q t\in[0,T],\\
\Si(0) &= 0,
\end{aligned}\right.\end{equation}
where $x\in\dbR^n$ is the initial state of the state equation \rf{state}, and
\begin{equation}\label{def:M&N}
  M(t) \deq D(t)D(t)^\top, \q N(t) \deq K(t)K(t)^\top, \q t\in[0,T].
\end{equation}
By reversing time,
$$\t = T-t, \q t\in[0,T], $$
one sees that \rf{Si} can be converted into a Riccati equation of the form \rf{Ric:P} and hence admits
a unique positive semidefinite solution $\Si\in C([0,T];\dbS^n)$ under the assumptions \ref{A1} and \ref{A2}.

\begin{theorem}\label{thm:SDE-hX}
Let \ref{A1}--\ref{A2} hold. For a fixed control $u\in\cU$, let $X$ be the corresponding state process
with initial state $x$ and $Y$ be the observation process. If $u$ is admissible, then $\hX$ evolves according to
the following SDE:
\begin{equation}\label{SDE:hX}\left\{\begin{aligned}
d\hX(t) &= [A(t)\hX(t)+B(t)u(t)+a(t)]dt + [\Si(t)H(t)^\top\! + C(t)K(t)^\top]N(t)^{-1} dV(t), \\
 \hX(0) &= x,
\end{aligned}\right.\end{equation}
where $\Si$ is the solution of \rf{Si}, and $V=\{V(t);\,0\le t\le T\}$ is the {\it innovation process} defined by
\begin{equation}\label{innovation}
 V(t) \deq Y(t) - \int_0^t \big[ H(s)\h X(s)+h(s) \big]ds, \q 0\le t\le T.
\end{equation}
\end{theorem}

\begin{corollary}\label{crllry:SDE-tX}
Let \ref{A1}--\ref{A2} hold. If $u\in\cU_{ad}$, then the difference $\tX \deq X-\hX$ satisfies the following SDE:
\begin{equation}\label{SDE:tX}\left\{\begin{aligned}
 d\tX(t) &= [A-(\Si H^\top\! + CK^\top) N^{-1}H ]\tX dt - \Si (K^{-1}H)^\top dW + DdW', \\
  \tX(0) &= 0.
\end{aligned}\right.\end{equation}
Further, $\Si(t)=\dbE[\tX(t)\tX(t)^\top]$.
\end{corollary}

We observe from \autoref{crllry:SDE-tX} that the process  $\tX$ is independent of the control $u$ and the initial state $x$.
Moreover, by the discussion at the end of \autoref{Sec:Pre}, $\tX$ is orthogonal to $\hX$, i.e.,
$$ \dbE\lan\tX(t),\hX(t)\ran=0, \q \dbE|X(t)|^2 = \dbE|\hX(t)|^2 + \dbE|\tX(t)|^2, \q t\in[0,T]. $$
This suggests the following decomposition of the cost functional \rf{cost}:
$$ J(x;u) = \hJ(x;u) + \tJ, $$
where $\hJ(x;u)$ and  $\tJ$ are given by
\begin{align}\label{def:hJ}
\hJ(x;u) &= \dbE\Big\{\lan G\hX(T),\hX(T)\ran + 2\lan g,\hX(T)\ran  \nn\\
&\hp{=\ } +\int_0^T\bigg[\Blan\begin{pmatrix*}[l]Q(t) & \!S(t)^\top \\ S(t) & \!R(t)\end{pmatrix*}\!
                              \begin{pmatrix}\hX(t) \\ u(t)\end{pmatrix}\!,
                              \begin{pmatrix}\hX(t) \\ u(t)\end{pmatrix}\Bran  \nn\\
&\hp{=\ } +2\Blan\begin{pmatrix}q(t) \\ r(t)\end{pmatrix}\!,
                 \begin{pmatrix}\hX(t) \\ u(t)\end{pmatrix}\Bran\bigg] dt\bigg\},
\end{align}
and
\begin{align}\label{def:tJ}
\tJ &= \dbE\bigg\{\lan G\tX(T),\tX(T)\ran + 2\lan g,\tX(T)\ran \nn\\
&\hp{=\ } +\int_0^T\[\lan Q(t)\tX(t),\tX(t)\ran +2\lan q(t),\tX(t)\ran\] dt\bigg\},
\end{align}
respectively. Note that $\tJ$ does not depend on the control $u$.

\ms

To deal with $\hJ(x;u)$, we recall from \autoref{lmm:LQ-theory} that under our assumptions \ref{A1}--\ref{A3},
the Riccati equation \rf{Ric:P} admits a unique positive semidefinite solution $P\in C([0,T];\dbS^n)$. Let
\begin{equation}\label{def-Theta}
  \Th(t) \deq -R(t)^{-1}[B(t)^\top P(t) + S(t)], \q t\in[0,T].
\end{equation}
The following ODE obviously has a unique solution $\f$:
\begin{equation}\label{ODE:phi}\left\{\begin{aligned}
&\dot\f(t) + [A(t)+B(t)\Th(t)]^\top\f(t) + \Th(t)^\top r(t) +P(t)a(t)+q(t)=0, \q t\in[0,T],\\
&\f(T)=g.
\end{aligned}\right.\end{equation}
Further, we write
\begin{equation}\label{notation}\left\{\begin{aligned}
 C(t) &= (C_1(t),\cdots,C_d(t)), \\
 D(t) &= (D_1(t),\cdots,D_k(t)), \\
\D(t) &= (\D_1(t),\cdots,\D_d(t)) \deq \Si(t)[K(t)^{-1}H(t)]^\top.
\end{aligned}\right.\end{equation}

\begin{proposition}\label{prop:hJ=}
Let \ref{A1}--\ref{A3} hold. For each $u\in\cU_{ad}$,
\begin{align}\label{hJ=P+phi}
\hJ(x;u)
&= \lan P(0)x,x\ran+2\lan\f(0),x\ran \nn\\
&\hp{=\ } +\dbE\int_0^T\bigg\{\lan R[\Th\hX-R^{-1}(B^\top\f+r)-u],[\Th\hX-R^{-1}(B^\top\f+r)-u]\ran \nn\\
&\hp{=\ } -\lan R^{-1}(B^\top\f+r),B^\top\f+r\ran + 2\lan\f,a\ran + \sum_{i=1}^d\lan P(\D_i+C_i),\D_i+C_i\ran\bigg\}dt.
\end{align}
\end{proposition}

To deal with $\tJ$, we let
\begin{equation}\label{cA}
 \cA(t) \deq A(t)-[\Si(t) H(t)^\top\! + C(t)K(t)^\top] N(t)^{-1}H(t), \q t\in[0,T],
\end{equation}
and further introduce the linear ODEs
\begin{equation}\label{Ric:Pi}\left\{\begin{aligned}
& \dot\varPi(t) + \varPi(t)\cA(t)  + \cA(t)^\top\varPi(t) + Q(t) =0, \\
& \varPi(T)=G
\end{aligned}\right.\end{equation}
and
\begin{equation}\label{ODE:pi}\left\{\begin{aligned}
&\dot\pi(t) + \cA(t)^\top\pi(t) + q(t) =0, \\
&\pi(T)=g.
\end{aligned}\right.\end{equation}

\begin{proposition}\label{prop:tJ=}
Let \ref{A1}--\ref{A3} hold. Then
\begin{align}\label{tJ=}
\tJ &= \int_0^T\[\sum_{i=1}^k\lan\varPi(t) D_i(t),D_i(t)\ran +\sum_{i=1}^d\lan\varPi(t)\D_i(t),\D_i(t)\ran\]dt.
\end{align}
\end{proposition}

Finally, combining the previous results, we can obtain the following theorem, which gives the optimal control
and the value function of Problem (O).

\begin{theorem}\label{thm:opt-u}
Let \ref{A1}--\ref{A3} hold. For each initial state $x\in\dbR^n$, Problem (O) admits a unique optimal control,
given by the following observation-feedback form:
\begin{align}\label{opt:formula}
u^*(t) &= \Th(t)\hX(t)-R(t)^{-1}[B(t)^\top\f(t)+r(t)].
\end{align}
The optional projection $\hX$ evolves according to the following SDE:
\begin{equation}\label{optimal:hX}\left\{\begin{aligned}
d\hX(t) &= \big\{[A(t)+B(t)\Th(t)]\hX(t)-B(t)R(t)^{-1}[B(t)^\top\f(t)+r(t)]+a(t)\big\}dt \\
&\hp{=\ } + [\Si(t)H(t)^\top\! + C(t)K(t)^\top] N(t)^{-1} \big\{dY(t)-[H(t)\hX(t)+h(t)]dt\big\}, \\
\hX(0) &= x.
\end{aligned}\right.\end{equation}
Moreover, the optimal value is given by
\begin{align}\label{opt:value}
J(x;u^*)
&= \lan P(0)x,x\ran+2\lan\f(0),x\ran +\int_0^T\bigg\{\sum_{i=1}^k\lan\varPi(t) D_i(t),D_i(t)\ran \nn\\
&\hp{=\ } +\sum_{i=1}^d\lan\varPi(t)\D_i(t),\D_i(t)\ran + \sum_{i=1}^d\lan P(t)[\D_i(t)+C_i(t)],\D_i(t)+C_i(t)\ran\nn\\
&\hp{=\ } -\lan R(t)^{-1}[B(t)^\top\f(t)+r(t)],B(t)^\top\f(t)+r(t)\ran + 2\lan\f(t),a(t)\ran\bigg\} dt.
\end{align}
\end{theorem}

\section{The filtering equation}\label{sec:filter}

Suppose that $u\in\cU$ is an admissible control. Let $X$ be the corresponding state process with initial state $x$
and Y the observation process. In this section we establish the filtering equation for
$$ \hX(t)=\dbE[X(t)\nid\sY_t^u], $$
as well as \autoref{crllry:SDE-tX}. For natational simplicity, we let
$$\tX(t) \deq X(t)-\hX(t), \q \Si(t) \deq \dbE[\tX(t)\tX(t)^\top]. $$
It will be shown later that the function $\Si$ defined above is the solution of the ODE \rf{Si}.

\ms

Let $V=\{V(t);\,0\le t\le T\}$ be the innovation process defined by \rf{innovation}:
$$ V(t) \deq Y(t) - \int_0^t \big[ H(s)\h X(s)+h(s) \big]ds. $$
First, we have the following result.

\begin{lemma}\label{lmm:hatV-BM}
Let \ref{A1}--\ref{A2} hold. For a fixed admissible control $u\in\cU_{ad}$,
the process $\check V=\{\check V(t);\,0\le t\le T\}$ defined by
$$ \check V(t) \deq \int_0^t K(s)^{-1}dV(s) $$
is a standard $\{\sY_t^u\}$-Brownian motion in $\dbR^d$.
\end{lemma}

\begin{proof}
Clearly, $\check V$ is a continuous, $\{\sY_t^u\}$-adapted, integrable process, and
\begin{align*}
d\check V(t)
&= K(t)^{-1}dV(t) = K(t)^{-1}H(t)[X(t)-\h X(t)]dt + dW(t) \\
&= K(t)^{-1}H(t)\tX(t)dt + dW(t).
\end{align*}
Thus, for $0\le s<t\le T$,
\begin{align*}
\check V(t)-\check V(s) &= W(t)-W(s) + \int_s^t K(\t)^{-1}H(\t)\tX(\t)d\t.
\end{align*}
Since for $\t\ge s$, $\tX(\t)$ is independent of $\sY_s^u$,
\begin{align*}
\dbE[\check V(t)-\check V(s)\nid \sY_s^u]
&= \dbE\[\dbE[W(t)-W(s)\nid\sF_s]\nid\sY_s^u\] + \dbE\lt[\int_s^t K(\t)^{-1}H(\t)\tX(\t)d\t \Bid\sY_s^u\rt] \\
&=\int_s^t K(\t)^{-1}H(\t)\dbE[\tX(\t)]d\t =0.
\end{align*}
This shows that $\check V$ is a $\{\sY_t^u\}$-martingale. Further,
\begin{align*}
d[\check V(t)\check V(t)^\top]
&=  K(t)^{-1}H(t)\tX(t)\check V(t)^\top dt + [dW(t)]\check V(t)^\top \\
&\hp{=\ } + \check V(t)[K(t)^{-1}H(t)\tX(t)]^\top dt + \check V(t)dW(t)^\top + I_d\,dt,
\end{align*}
and hence
\begin{align}\label{checkV}
&\check V(t)\check V(t)^\top-\check V(s)\check V(s)^\top
=  \int_s^t K(\t)^{-1}H(\t)\tX(\t)\check V(\t)^\top d\t + \lt[\int_s^t\check V(\t)dW(\t)^\top\rt]^\top \nn\\
&\qq\q + \lt[\int_s^t K(\t)^{-1}H(\t)\tX(\t)\check V(\t)^\top d\t\rt]^\top + \int_s^t \check V(\t)dW(\t)^\top + (t-s)I_d.
\end{align}
Again, since $\tX(\t)$ is independent of $\sY_s^u$ for $\t\ge s$, we have
\begin{align*}
& \dbE\lt[\int_s^t K(\t)^{-1}H(\t)\tX(\t)\check V(\t)^\top d\t\Bid \sY_s^u\rt] \\
&\q=\int_s^t K(\t)^{-1}H(\t)\dbE\[\tX(\t)\check V(\t)^\top\bid \sY_s^u\] d\t \\
&\q=\int_s^t K(\t)^{-1}H(\t)\dbE\Big\{\dbE\big[\tX(\t)\check V(\t)^\top\bid \sY_\t^u\big]\bid \sY_s^u\Big\} d\t \\
&\q=\int_s^t K(\t)^{-1}H(\t)\dbE\Big\{\dbE[\tX(\t)]\check V(\t)^\top\bid \sY_s^u\Big\} d\t \\
&\q=0.
\end{align*}
For the stochastic integral in \rf{checkV}, we have
$$\dbE\lt[\int_s^t \check V(\t)dW(\t)^\top\Bid \sY_s^u\rt] = \dbE\lt\{\dbE\lt[\int_s^t \check V(\t)dW(\t)^\top\Bid \sF_s \rt]\bigg| \sY_s^u\rt\} = 0. $$
So taking conditional expectations with respect to $\sY_s^u$ on both sides of \rf{checkV} yields
$$ \dbE\[\check V(t)\check V(t)^\top-\check V(s)\check V(s)^\top \bid \sY_s^u\] = (t-s)I_d, $$
which implies that the cross-variations are given by
$$ \lan\check V_i,\check V_j\ran_t = \d_{ij}t; \q 1\le i, j\le d. $$
It then follows from
the martingale characterization of Brownian motion that $\check V$ is a standard $d$-dimensional Brownian motion.
\end{proof}

Now we consider the process $\L=\{\L(t);\,0\le t\le T\}$ defined by
\begin{equation}\label{def:Lambda}
\L(t) \deq \hX(t)-x-\int_0^t \big[A(s)\hX(s)+B(s)u(s)+a(s)\big] ds.
\end{equation}
Clearly, it is $\{\sY_t^u\}$-adapted and RCLL, with $\L(0)=0$ almost surely.
Further, by \rf{Bound:XY} it is also square-integrable.
The following result shows that $\L$ is actually a $\{\sY_t^u\}$-martingale.

\begin{lemma}\label{lmm:Lambda}
The process $\L=\{\L(t);\,0\le t\le T\}$ defined by \rf{def:Lambda} is an RCLL, square-integrable
$\{\sY_t^u\}$-martingale with $\L(0)=0$ a.s..
\end{lemma}

\begin{proof}
For fixed $0\le s<t\le T$, we have
\begin{align*}
&\dbE\[\L(t)-\L(s) \nid \sY_s^u\] \\
&\q= \dbE\[\hX(t)-\hX(s) \nid \sY_s^u\] - \dbE\lt\{\int_s^t \big[A(\t)\hX(\t)+B(\t)u(\t)+a(\t)\big] d\t \Bid \sY_s^u\rt\} \\
&\q= \dbE\[X(t)-X(s) \nid \sY_s^u\] - \int_s^t \dbE\[A(\t)\hX(\t)+B(\t)u(\t)+a(\t)\bid \sY_s^u\] d\t  \\
&\q= \dbE\[X(t)-X(s) \nid \sY_s^u\] - \int_s^t \dbE\[A(\t)X(\t)+B(\t)u(\t)+a(\t)\bid \sY_s^u\] d\t  \\
&\q= \dbE\[X(t)-X(s) \nid \sY_s^u\] - \dbE\lt\{\int_s^t \big[A(\t)X(\t)+B(\t)u(\t)+a(\t)\big]d\t \Bid \sY_s^u\rt\} \\
&\q= \dbE\lt[\int_s^t C(\t)dW(\t) + \int_s^t D(\t)dW'(\t) \Bid \sY_s^u\rt]\\
&\q= \dbE\lt\{\dbE\lt[\int_s^t C(\t)dW(\t) + \int_s^t D(\t)dW'(\t) \Bid \sF_s\rt] \bigg| \sY_s^u\rt\} \\
&\q= 0.
\end{align*}
This shows that $\L$ is a $\{\sY_t^u\}$-martingale.
\end{proof}

It is not clear whether the process $\L$ is adapted to the smaller filtration generated by the Brownian motion $\check{V}$.
So we cannot conclude directly by the martingale representation theorem that $L$ can be expressed as a stochastic
integral with respect to $\check{V}$.
Fortunately, thanks to the theorem of Fujisaki, Kallianpur and Kunita (see Rogers and Williams \cite{Rogers-Williams2000}, VI.8),
there exists an $\dbR^{n\times d}$-valued, square-integrable, $\{\sY_t^u\}$-progressively measurable process
$\l=\{\l(t);\,0\le t\le T\}$ such that
\begin{equation}\label{M-lambda}
   \L(t) = \int_0^t \l(s)N(s)^{-1} dV(s), \q 0\le t\le T.
\end{equation}
Next we determine the process $\l$.

\begin{proposition}\label{lmm:lambda}
The process $\l$ in \rf{M-lambda} is given by
\begin{equation}\label{K-filter:K}
   \l(t) = \Si(t)H(t)^\top\! + C(t)K(t)^\top, \q 0\le t\le T,
\end{equation}
where $\Si(t)\deq\dbE[\tX(t)\tX(t)^\top]$.
\end{proposition}

\begin{proof}
Let $\z=\{\z(t);\,0\le t\le T\}$  be a fixed but arbitrary $\dbR^{n\times d}$-valued, square-integrable,
$\{\sY_t^u\}$-progressively measurable processes.
Consider the $\{\sY_t^u\}$-martingale
$$\eta(t) = \int_0^t \z(s)N(s)^{-1}dV(s), \q 0\le t\le T.$$
Using \rf{M-lambda} we have
\begin{align}\label{E(phieta)-1}
\dbE\[\L(t)\eta(t)^\top\] &= \dbE\int_0^t \l(s)N(s)^{-1}\z(s)^\top ds, \q\forall 0\le t\le T.
\end{align}
On the other hand, using \rf{def:Lambda} we have
\begin{align*}
\dbE\[\L(t)\eta(t)^\top\]
&= \dbE\[\hX(t)\eta(t)^\top\] -\int_0^t \dbE\Big\{\big[A(s)\hX(s)+B(s)u(s)+a(s)\big]\eta(t)^\top\Big\} ds.
\end{align*}
Since $\{\eta(t);\,0\le t\le T\}$ is a $\{\sY_t^u\}$-martingale, we have for $0\le s\le t\le T$,
\begin{align*}
\dbE\[\hX(s)\eta(t)^\top\]
&= \dbE\[\dbE\big[\hX(s)\eta(t)^\top\bid \sY_s^u\big]\] = \dbE\[\hX(s)\dbE\big[\eta(t)^\top\nid \sY_s^u\big]\] \\
&= \dbE\[\hX(s)\eta(s)^\top\] = \dbE\[\dbE[X(s)\nid \sY_s^u]\eta(s)^\top\]  \\
&= \dbE\[X(s)\eta(s)^\top\].
\end{align*}
Similarly, for $0\le s\le t\le T$,
\begin{align*}
\dbE\Big\{[B(s)u(s)+a(s)]\eta(t)^\top\Big\} = \dbE\Big\{[B(s)u(s)+a(s)]\eta(s)^\top\Big\}.
\end{align*}
It follows that
\begin{align}\label{E(phieta)-2}
\dbE\[\L(t)\eta(t)^\top\]
&= \dbE\[X(t)\eta(t)^\top\] -\dbE\int_0^t \big[A(s)X(s)+B(s)u(s)+a(s)\big]\eta(s)^\top ds.
\end{align}
We observe that
\begin{align*}
d\eta(t)
&= \z(t)N(t)^{-1}dV(t) = \z(t)N(t)^{-1}\Big\{ dY(t) -[ H(t)\h X(t)+h(t)]dt \Big\} \\
&= \z(t)N(t)^{-1}\[K(t)dW(t)+H(t)\tX(t)dt\] \\
&= \z(t)N(t)^{-1}H(t)\tX(t)dt + \z(t)N(t)^{-1}K(t)dW(t).
\end{align*}
Thus, integration by parts yields
\begin{align}\label{EXeta}
\dbE[X(t)\eta(t)^\top]
&=\dbE\int_0^t [A(s)X(s)+B(s)u(s)+a(s)]\eta(s)^\top ds \nn\\
&\hp{=\ } +\dbE\int_0^t X(s)\tX(s)^\top H(s)^\top N(s)^{-1}\z(s)^\top ds \nn\\
&\hp{=\ } + \dbE\int_0^t C(s)K(s)^\top N(s)^{-1}\z(s)^\top ds.
\end{align}
Using the facts
$$ \z(s)\in \sY_s^u; \q \tX(s)~\text{is independent of}~\sY_s^u; \q \dbE[\tX(s)]=0, $$
we see that
\begin{align*}
& \dbE\int_0^t X(s)\tX(s)^\top H(s)^\top N(s)^{-1}\z(s)^\top ds \\
&~= \dbE\!\int_0^t \tX(s)\tX(s)^\top H(s)^\top N(s)^{-1}\z(s)^\top ds
    +\dbE\!\int_0^t \hX(s)\tX(s)^\top H(s)^\top N(s)^{-1}\z(s)^\top ds  \\
&~= \dbE\!\int_0^t \dbE[\tX(s)\tX(s)^\top] H(s)^\top N(s)^{-1}\z(s)^\top ds
    +\dbE\!\int_0^t \hX(s)\dbE[\tX(s)^\top] H(s)^\top N(s)^{-1}\z(s)^\top ds \\
&~= \dbE\!\int_0^t \Si(s)H(s)^\top N(s)^{-1}\z(s)^\top ds.
\end{align*}
Now we can obtain from \rf{EXeta} that
\begin{align*}
\dbE[X(t)\eta(t)^\top]
&=\dbE\int_0^t \big[A(s)X(s)+B(s)u(s)+a(s)\big]\eta(s)^\top ds \\
&\hp{=\ } +\dbE\int_0^t \[\Si(s)H(s)^\top+ C(s)K(s)^\top\] N(s)^{-1}\z(s)^\top ds.
\end{align*}
The above, together with \rf{E(phieta)-2}, gives
\begin{align}\label{E(phieta)-3}
\dbE\[\L(t)\eta(t)^\top\]
&= \dbE\int_0^t \[\Si(s)H(s)^\top+C(s)K(s)^\top\] N(s)^{-1}\z(s)^\top  ds.
\end{align}
Comparing \rf{E(phieta)-1} and \rf{E(phieta)-3} and noting that $\z$ is arbitrary, we get \rf{K-filter:K}.
\end{proof}

We are ready now for the proofs of \autoref{thm:SDE-hX} and \autoref{crllry:SDE-tX}.

\begin{proof}[\textbf{Proofs of \autoref{thm:SDE-hX} and \autoref{crllry:SDE-tX}}]
From the definition \rf{def:Lambda} of $\L$, the representation \rf{M-lambda} of $\L$, and \autoref{lmm:lambda},
it follows that
\begin{align}\label{hX:expanded-1}
d\hX(t) &= [A(t)\hX(t)+B(t)u(t)+a(t)]dt +  [\Si(t) H(t)^\top\! + C(t)K(t)^\top] N(t)^{-1}dV(t).
\end{align}
By the definition of $V$ and $Y$,
$$ dV(t) = dY(t) - [H(t)\hX(t) + h(t)]dt = H(t)\tX(t)dt + K(t)dW(t), $$
which, substituted in \rf{hX:expanded-1}, yields
\begin{align}\label{hX:expanded-2}
d\hX(t) &= (A\hX+Bu+a)dt + (\Si H^\top\!+CK^\top) N^{-1}H\tX dt  \nn\\
&\hp{=\ } + (\Si H^\top\!+CK^\top) N^{-1}K dW(t).
\end{align}
Noting that $N(t)=K(t)K(t)^\top$, we obtain by subtracting \rf{hX:expanded-2} from the SDE of $X$ that
\begin{align*}
 d\tX(t) &= [A - (\Si H^\top\!+CK^\top) N^{-1}H ]\tX dt - \Si(K^{-1}H)^\top dW + DdW'.
\end{align*}
Integration by parts then gives
\begin{align*}
\Si(t)- \Si(0)
&= \dbE[\tX(t)\tX(t)^\top] - \dbE[\tX(0)\tX(0)^\top] \\
&= \dbE\int_0^t [A(s) - (\Si(s)H(s)^\top\!+C(s)K(s)^\top) N(s)^{-1}H(s)]\tX(s)\tX(s)^\top ds \\
&\hp{=\ } + \dbE\int_0^t \tX(s)\tX(s)^\top[A(s) - (\Si(s)H(s)^\top\!+C(s)K(s)^\top) N(s)^{-1}H(s)]^\top ds \\
&\hp{=\ } + \dbE\int_0^t \Si(s)[K(s)^{-1}H(s)]^\top K(s)^{-1}H(s)\Si(s)ds + \dbE\int_0^t D(s)D(s)^\top ds \\
&= \int_0^t [A(s) - (\Si(s)H(s)^\top\!+C(s)K(s)^\top) N(s)^{-1}H(s)]\Si(s) ds \\
&\hp{=\ } + \int_0^t \Si(s)[A(s) - (\Si(s)H(s)^\top\!+C(s)K(s)^\top) N(s)^{-1}H(s)]^\top ds \\
&\hp{=\ } + \int_0^t [\Si(s)H(s)^\top N(s)^{-1}H(s)\Si(s) + M(s)] ds  \\
&= \int_0^t \Big\{[A(s)-C(s)K(s)^{-1}H(s)]\Si(s) + \Si(s)[A(s)-C(s)K(s)^{-1}H(s)]^\top  \\
&\hp{=\ } -\Si(s) H(s)^\top N(s)^{-1}H(s)\Si(s) + M(s)\Big\} ds,
\end{align*}
which is exactly the integral version of \rf{Si}.
\end{proof}

\section{The optimal control}\label{sec:optimal-u}

This section is devoted to the proofs of \autoref{prop:hJ=}, \autoref{prop:tJ=} and \autoref{thm:opt-u}.
Recall that if the control $u\in\cU$ is admissible, then $\tX$ is orthogonal to $\hX$ and thereby
we can write the cost functional as
$$ J(x;u) = \hJ(x;u) + \tJ, $$
where $\hJ(x;u)$ and $\tJ$ are given by \rf{def:hJ} and \rf{def:tJ}, respectively.
With the notation \rf{notation}, we see that
\begin{equation}
[\Si(t) H(t)^\top\!+C(t)K(t)^\top][K(t)^\top]^{-1} = (\D_1(t)+C_1(t),\cdots,\D_d(t)+C_d(t)).
\end{equation}
We prove \autoref{prop:hJ=} first.

\begin{proof}[\textbf{Proof of \autoref{prop:hJ=}.}]
For an admissible control $u\in\cU_{ad}$, we know from \autoref{thm:SDE-hX} that $\hX$,
the $\{\sY_t^u\}$-optional projection of the state process $X$, evolves according to the SDE \rf{SDE:hX}.
Let $\check{V}_i(t)$ be the $i$th component of $\check{V}(t)$ so that
\begin{equation*}
d\hX(t) = [A(t)\hX(t)+B(t)u(t)+a(t)]dt + \sum_{i=1}^d [\D_i(t)+C_i(t)] d\check{V}_i(t).
\end{equation*}
Recall from \autoref{lmm:hatV-BM} that the processes $\check{V}_i$; $1\le i\le d$ are independent Brownian motions.
We obtain by applying It\^{o}'s formula to $t\mapsto\lan P(t)\hX(t),\hX(t)\ran$ that
\begin{align*}
& \dbE\lan G\hX(T),\hX(T)\ran - \lan P(0)x,x\ran \\
&\q= \dbE\int_0^T\[\lan\dot P\hX,\hX\ran + 2\lan P\hX,A\hX+Bu+a\ran + \sum_{i=1}^d\lan P(\D_i+C_i),\D_i+C_i\ran\]dt.
\end{align*}
Furthermore, applying It\^{o}'s formula to $t\mapsto\lan \f(t),\hX(t)\ran$ yields
\begin{align*}
& \dbE\lan g,\hX(T)\ran-\lan\f(0),x\ran
    = \dbE\int_0^T\[\lan\dot\f,\hX\ran + \lan \f,A\hX+Bu+a\ran \]dt.
\end{align*}
Substituting for $\dbE\lan G\hX(T),\hX(T)\ran$ and $\dbE\lan g,\hX(T)\ran$ in the cost functional $\hJ(x;u)$,
then using \rf{Ric:P} and \rf{ODE:phi} to simplify the computation, we obtain \rf{hJ=P+phi}.
\end{proof}

Next we prove \autoref{prop:tJ=}.

\begin{proof}[\textbf{Proof of \autoref{prop:tJ=}.}]
Recall \rf{notation} and \rf{cA}--\rf{ODE:pi}, and note that the SDE \rf{SDE:tX} can be written as
\begin{equation}\label{SDE:tX-plus}\left\{\begin{aligned}
 d\tX(t) &= \cA(t)\tX(t) dt - \sum_{i=1}^d\D_i(t)dW_i(t) + \sum_{i=1}^k D_i(t)dW'_i(t), \\
  \tX(0) &= 0.
\end{aligned}\right.\end{equation}
Applying It\^{o}'s formula to $t\mapsto\lan\varPi(t)\tX(t),\tX(t)\ran$ yields
\begin{align}\label{tJ:sub1}
& \dbE\lan G\tX(T),\tX(T)\ran \nn\\
&\q= \dbE\int_0^T\[\lan\dot\varPi\tX,\tX\ran + 2\lan\varPi\tX,\cA\tX\ran +\sum_{i=1}^d\lan\varPi\D_i,\D_i\ran
     + \sum_{i=1}^k\lan\varPi D_i,D_i\ran \]dt \nn\\
&\q= \dbE\int_0^T\[\sum_{i=1}^d\lan\varPi\D_i,\D_i\ran + \sum_{i=1}^k\lan\varPi D_i,D_i\ran - \lan Q\tX,\tX\ran\]dt.
\end{align}
Applying It\^{o}'s formula to $t\mapsto\lan\pi(t),\tX(t)\ran$ yields
\begin{align}\label{tJ:sub2}
\dbE\lan g,\tX(T)\ran &= \dbE\int_0^T\[\lan\dot\pi,\tX\ran + \lan \pi,\cA\tX\ran \]dt = -\dbE\int_0^T \lan q,\tX\ran dt.
\end{align}
Substitution of \rf{tJ:sub1} and \rf{tJ:sub2} into \rf{def:tJ} results in \rf{tJ=}.
\end{proof}

Finally, we present the proof of \autoref{thm:opt-u}.

\begin{proof}[\textbf{Proof of \autoref{thm:opt-u}.}]
For a fixed admissible control $u\in\cU_{ad}$, let $X$ be the corresponding state process
with initial state $x$ and $Y$ the observation process.
By \autoref{thm:SDE-hX}, the $\{\sY_t^u\}$-optional projection $\hX$ of $X$ evolves according to the SDE \rf{SDE:hX},
and by \autoref{crllry:SDE-tX}, the difference process $\tX$ evolves according to the SDE \rf{SDE:tX}.
The cost functional can be written as the sum of $\hJ(x;u)$ and $\tJ$, given by \rf{def:hJ} and \rf{def:tJ}, respectively.
Observe that the process $\tX$ does not depend on the choice of $u\in\cU_{ad}$, and hence neither does $\tJ$.
So Problem (O) is equivalent to finding a control $v\in\sY^v$ such that $\hJ(x;u)$ is minimized over $u\in\cU_{ad}$.
On the other hand, according to \autoref{prop:hJ=},
\begin{align*}
\hJ(x;u)
&\ge \lan P(0)x,x\ran+2\lan\f(0),x\ran +\dbE\int_0^T\bigg[\sum_{i=1}^d\lan P(\D_i+C_i),\D_i+C_i\ran\nn\\
&\hp{=\ }-\lan R^{-1}(B^\top\f+r),B^\top\f+r\ran + 2\lan\f,a\ran\bigg]dt,
\end{align*}
with equality if and only if
$$ u(t) = \Th(t)\hX(t)-R(t)^{-1}[B(t)^\top\f(t)+r(t)]. $$
This means that the admissible control given by \rf{opt:formula} is optimal.
Substituting \rf{opt:formula} into the filtering equation \rf{SDE:hX} yields \rf{optimal:hX}.
Finally, adding $\hJ(x;u^*)$ and $\tJ$, we get the optimal value \rf{opt:value}.
\end{proof}

\section{Conclusion remarks}\label{sec:conclusion}

We have studied a class of linear-quadratic optimal control problems for partially observable dynamical systems.
Without imposing additional requirements on the admissible control, we showed that the optimal control is given
by a feedback representation via the filtering process and obtained the optimal value explicitly.
Our method is based on the orthogonal decomposition of the state process, which allows us to write the cost
functional as the sum of two independent parts: one depends only on the control and the filtering process
and the other is a functional of the estimate error independent of the choice of the control.
An important feature making our approach work is that the diffusion of the state equation does not involve
the state and the control.
With this feature, our idea could apply to more general models, for example, the mean-field model, the backward
optimal control problem for partially observable dynamical systems, and the optimal control problem of FBSDEs
with partial information.
We hope to report some results relevant to these problems in our future publications.

\end{document}